\documentclass[12pt]{amsart}
\usepackage[T1]{fontenc}
\usepackage[utf8]{inputenc}

\usepackage[american]{babel}
\usepackage[babel, final]{microtype}
\usepackage[all]{xy}
\usepackage{ifpdf}
\usepackage{multirow}

\usepackage[pdftex, dvipsnames]{xcolor}
\usepackage[pdftex, final]{graphicx}
\usepackage{pinlabel}
\usepackage[pdftex,%
    nomarginpar,%
    lmargin=1in,%
    rmargin=1in,%
    tmargin=1in,%
    bmargin=1in,%
]{geometry}

\usepackage{amsmath,amsthm,amssymb,amsfonts,mathrsfs,tikz,tikz-cd,bm,verbatim,mathtools,caption,enumitem,csquotes,float,setspace,subcaption}

\usepackage[normalem]{ulem}

\usepackage{hyperref}
\hypersetup{
    colorlinks=true,
    linkcolor=NavyBlue,
    citecolor=NavyBlue
}

\usepackage[hang,flushmargin]{footmisc} 

\usepackage[backend=biber,style=alphabetic,sorting=nyt, url=false,doi=false,isbn=false,backref=true,maxnames=10]{biblatex}
\renewbibmacro{in:}{}
\DefineBibliographyStrings{english}{
  backrefpage={$\uparrow$},
  backrefpages={$\uparrow$}
}
\AtBeginBibliography{\small}
\addto\extrasamerican{%
}

\numberwithin{equation}{section}

\newtheorem{theorem}{Theorem}[section]
\newtheorem*{theorem*}{Theorem}
\newtheorem*{framework*}{Framework for computing $\mathbf {\mathcal S}(X)$ for trisections}
\newtheorem{corollary}{Corollary}[section]

\newtheorem{lemma}{Lemma}[section]

\theoremstyle{definition}
\newtheorem{definition}{Definition}[section]

\newtheorem{remark}{Remark}[section]

\makeatletter
\let\c@conjecture=\c@theorem
\let\c@corollary=\c@theorem
\let\c@proposition=\c@theorem
\let\c@lemma=\c@theorem
\let\c@definition=\c@theorem
\let\c@example=\c@theorem
\let\c@remark=\c@theorem
\let\c@equation\c@theorem
\let\c@question\c@theorem
\let\c@fact\c@theorem
\makeatother

\def\makeautorefname#1#2{\expandafter\def\csname#1autorefname\endcsname{#2}}
\makeautorefname{proposition}{Proposition}
\makeautorefname{lemma}{Lemma}
\makeautorefname{definition}{Definition}
\makeautorefname{example}{Example}
\makeautorefname{question}{Question}
\makeautorefname{conjecture}{Conjecture}
\makeautorefname{section}{Section}
\makeautorefname{claim}{Claim}
\makeautorefname{equation}{Equation}
\makeautorefname{remark}{Remark}
\makeautorefname{corollary}{Corollary}
\makeautorefname{fact}{Fact}

\DeclareMathOperator{\id}{id}

\DeclareMathOperator{\Obj}{Obj}

\DeclareMathOperator{\Hom}{Hom}

\usepackage[commandnameprefix=always]{changes}
\usepackage{xcolor}

\title{Cornered skein lasagna theory}

\author{Sarah Blackwell}
\address{University of Virginia, Charlottesville, United States}
\email{\href{mailto:blackwell@virginia.edu}{blackwell@virginia.edu}}
\urladdr{\url{https://seblackwell.com}}

\author{Vyacheslav Krushkal}
\address{University of Virginia, Charlottesville, United States}
\email{\href{mailto:krushkal@virginia.edu}{krushkal@virginia.edu}}
\urladdr{}

\author{Yangxiao Luo}
\address{University of Virginia, Charlottesville, United States}
\email{\href{mailto:yl8by@virginia.edu}{yl8by@virginia.edu}}
\urladdr{}

\begin{document}
\begin{abstract}
We extend the skein lasagna theory of Morrison–Walker–Wedrich to $4$-manifolds with corners and formulate gluing formulas for $4$-manifolds with boundary and, more generally, with corners. As an application, we
develop a categorical framework for a presentation of the skein lasagna module of trisected closed $4$-manifolds.
Further, we extend the theory to dimension two by introducing bicategories for closed oriented surfaces and proving a gluing formula for the categories associated with $3$-manifolds with boundary.
\end{abstract}
\maketitle

\section{Introduction}
The skein lasagna module $\mathcal{S}^N(X,L)$ is an invariant of a smooth $4$-manifold with a framed link $L$ in its boundary, introduced by Morrison, Walker and Wedrich \cite{MWW2}. It takes as input the $\mathfrak{gl}_N$ Khovanov-Rozansky homology \cite{KR} of links in the $3$-sphere and provides a universal extension to $4$-manifolds using embedded surfaces in $X$ with boundary equal to $L$ and a collection of links in the boundary of a finite family of $4$-balls removed from the interior of $X$; see \autoref{sec: background} for more details. The skein lasagna module has a bigrading induced from the Khovanov-Rozansky homology, and it may be considered as the degree zero part of blob homology theory of Morrison-Walker \cite{MW}. 
Fixing $N$ throughout, we abbreviate $\mathcal{S}^N(X,L)$ as $\mathcal{S}(X,L)$.

Direct calculation of the skein lasagna module from its definition appears feasible only in the simplest cases: the $4$-sphere or the $4$-ball with a link $L$ in its boundary. The theory in these cases recovers the ground ring for $S^4$, and the Khovanov-Rozansky homology for $(D^4,L)$. A method for computing the lasagna module for $4$-manifolds built without $1$- or $3$-handles was introduced by Manolescu-Neithalath \cite{MN}. In this case it reformulates $\mathcal{S}(X)$ using the {\em cabled Khovanov-Rozansky} homology for the attaching link of the $2$-handles in $S^3$, a more concrete invariant, but still very challenging to compute. Additional methods that incorporate $1$- and $3$-handles were developed by Manolescu-Walker-Wedrich \cite{MWW1}.
It was shown by Sullivan-Zhang \cite{SZ} and by Ren-Willis \cite{RW} that $\mathcal{S}(S^2\times S^2)$ is trivial. Further vanishing results for $4$-dimensional $2$-handlebodies were established in \cite{RW} under certain assumptions on the framing of the $2$-handles.

In a major contribution, Ren and Willis \cite{RW} showed that the skein lasagna module is a powerful $4$-manifold invariant, distinguishing an exotic pair of compact smooth $4$-manifolds. It is therefore desirable to develop new methods of computing the invariant. The aim of this paper is to develop a TQFT framework for studying the skein lasagna module using cut-and-paste methods. 

In more detail, let $X$ be a $4$-manifold with corners with $\partial X \cong (-Y_1)\cup_{\Sigma}Y_2$, where $Y_1, Y_2$ are $3$-manifolds with common boundary equal to a closed surface $\Sigma$. The invariant of $X$, with a specified finite collection of points $P\subset \Sigma$,  takes the form of a bifunctor 
$$F_{X,P} \colon \mathcal{S}(Y_1,P) \times \mathcal{S}(Y_2,P)^{op} \to \mathcal{V}$$ defined on certain categories associated with $Y_1, Y_2$ and valued in the category $\mathcal{V}$ of $\mathbb{Z}^2$-graded $\mathbb{Z}$-modules; see \autoref{sec: gluing}.

Given two such manifolds $X_1, X_2$ with $\partial X_1\cong (-Y_1) \cup_{\Sigma} Y_2$ and $\partial X_2\cong (-Y_2) \cup_{\Sigma} Y_3$ we formulate a suitable version of the tensor product of their bifunctors $F_{X_1,P}, F_{X_2,P}$. The following theorem (\autoref{thm:theorem1} in \autoref{sec: gluing}) states our first gluing result.

\begin{theorem*}
In the notation as above, \begin{equation} \label{eq: gluing thm}
        F_{X_1, P} \otimes^{}_{\mathcal{S}(Y_2,P)} F_{X_2, P} \cong F_{X_1 \cup^{}_{Y_2} X_2, P}
\end{equation}
    as $(\mathcal{S}(Y_1, P), \mathcal{S}(Y_3, P))$-bimodules.
\end{theorem*}

As a special case, this result applies to gluing $4$-manifolds with boundary (without corners) as well. A version of \autoref{eq: gluing thm} is also established for {\em self-gluing}, where a $4$-manifold $X$ has boundary of the form 
$(-Y) \cup_{\Sigma} Y$ and the two copies of $Y$ are identified.

In addition to setting up a TQFT framework, our work was motivated by the goal of developing a new computational method of the skein lasagna module of trisected $4$-manifolds. In this case a closed $4$-manifold is given as the union of three cornered $4$-manifolds which are $4$-dimensional $1$-handlebodies $\natural^{k} (S^1 \times B^3)$, subject to certain additional conditions; see \autoref{subsec:trisections} for a more detailed discussion. The following statement summarizes in general terms our presentation of the skein lasagna module for a given trisection $X=X_1\cup X_2\cup X_3$; more precise theorems are given in  \autoref{sec: gluing}. 

\begin{framework*}
For a closed trisected $4$-manifold $X$ its skein lasagna module ${\mathcal S}(X)$ can be assembled from the cornered invariants of the $1$-handlebodies $X_i$ using our gluing formulas, together with a stabilization equivalence relation on marked points $P\subset \Sigma$ in the central surface.
\end{framework*}

In principle, this reduces the problem of computing the invariant of $X$ to understanding the cornered invariant of the standard pieces (4-dimensional $1$-handlebodies $X_i$) and how they interact under gluing.
However, it was shown in \cite{MWW1} that $\mathcal{S}(S^1 \times B^3, L)$ is locally infinitely generated (i.e. infinitely generated in a particular bidegree) for some links $L\subset \#^k(S^1\times S^2)$, so using gluing results such as \autoref{eq: gluing thm} in this case is likely to be highly non-trivial. Using this approach to calculations of the skein lasagna module of a trisected $4$-manifold with $k=0$ (and $\text{genus}(\Sigma)\geq1 $) is more feasible.

We extend the theory down to dimension two in \autoref{sec:other}, defining bicategories for closed orientable surfaces. A compact oriented $3$-manifold $Y$ with $\partial Y=\Sigma$ gives rise to a module over the bicategory associated with $\Sigma$, and it is shown in \autoref{thm: gluing 3-manifolds} that the invariant of a closed $3$-manifold $Y_1\cup_{\Sigma} Y_2$ is isomorphic to a suitably defined tensor product of the modules assigned to $Y_1$ and $Y_2$.

\begin{remark}
A cornered version of Heegaard Floer homology was developed in \cite{DM, DLM}. We note the difference in dimensions between our setting and the work in these references. Our top-dimensional gluing result applies to cornered $4$-manifolds, and we extend the theory down to dimension two. The gluing results in \cite{DLM} concern $3$-manifolds with codimension-$2$ corners, and the theory is extended to $1$-manifolds. 
\end{remark}

\subsection*{Acknowledgments} 
Our approach was strongly influenced by the ideas and TQFT methods developed in \cite{Walker} and \cite{FNWW}.

We are grateful to Paul Wedrich for numerous conversations and for generously sharing his insight into the subject. We also would like to thank Chris Douglas, Robert Lipshitz, and Mike Willis for helpful conversations. 

SB was supported by the NSF Postdoctoral Research Fellowship DMS-2303143. VK was supported in part by NSF Grant DMS-2405044.

\section{Background and definitions} \label{sec: background}

In this paper we work in the smooth category. Let $\Sigma$ be a fixed closed, oriented surface of genus $g$, and let $I$ denote the unit interval which is taken to be $[-1,1]$. We assume tangles are framed and properly embedded.

We first recall the relevant terminology from skein lasagna theory \cite{MW, MWW2}. Let $Z$ be 
a functorial link invariant in ${\mathbb R}^3$ satisfying the conditions in \cite[Theorem~2.1]{MWW3}.

\begin{definition}[\cite{MWW2}]
    Given an oriented $4$-manifold $X$ with an oriented framed link $L \subset \partial X$, a \textit{lasagna filling} $F$ of $(X,L)$ consists of $(S, \{(B_i, L_i, v_i)\})$, where
    \begin{itemize}
    \setlength{\itemsep}{5pt}
        \item[-] $\{B_i\}$ is a finite collection of disjoint $4$-balls embedded in the interior of $X$, called \textit{input balls},
        \item[-] $L_i$ is an oriented framed link in $\partial B_i$,
        \item[-] $v_i$ is an element in $Z(L_i)$, sometimes called a \textit{homogeneous label}, and
        \item[-] $S$ is an oriented framed surface properly embedded in $X \setminus \bigcup_i B_i$, such that $S \cap \partial X = L$ and $S \cap \partial B_i = L_i$ for each $i$.
    \end{itemize}
\end{definition}

\begin{definition}[\cite{MWW2}]
    Given an oriented $4$-manifold $X$ with an oriented framed link $L \subset \partial X$, its \textit{skein lasagna module} $\mathcal{S}(X,L)$ is the $R$-module generated by the lasagna fillings of $(X,L)$, modulo the transitive and linear closure of the following relation. 
    \begin{itemize}
    \setlength{\itemsep}{5pt}
        \item[-] Linear combinations of lasagna fillings are set to be multilinear in the labels $v_i$.
        \item[-] Two lasagna fillings $F_1$ and $F_2$ are set to be equivalent if $F_1$ has an input ball $B_1$ with label $v_1$, and $F_2$ is obtained from $F_1$ by replacing $B_1$ with a lasagna filling $F_3 = (S', \{(B_{j}, L_{j}, v_{j})\})$ of $(B_1, L_1)$, such that $v_1 = Z (S')(\bigotimes_j v_{j})$, followed by an isotopy rel. boundary. A depiction of this relation can be found in \cite[Section~2.2]{MWW2}.
    \end{itemize}
\end{definition}

Throughout this paper, we will assume that the TQFT $Z$ is the Khovanov-Rozansky $\mathfrak{gl}_N$ homology $KhR_N$, which takes values in $\mathbb{Z}^2$-graded $\mathbb{Z}$-modules. The \textit{bidegree} of a lasagna filling $F = (S, \{(B_i, L_i, v_i)\})$ is given by
$$
\deg(F) = \sum_i \deg(v_i) + (0, (1-N)\chi(S)) \in \mathbb{Z}^2.
$$
Then the skein lasagna module $\mathcal{S}(X,L)$ inherits a $\mathbb{Z}^2$-grading.

We now set up the categorical framework which we will use throughout the rest of the paper. Our work relies on TQFT methods (which apply in much greater generality) developed in \cite{Walker,FNWW}. 
First we discuss a relatively simple gluing map that will be useful for the definition of this category. Let $X_1, X_2$ be oriented $4$-manifolds and $L_1, L_2$ be framed oriented links in the boundary of $X_1, X_2$ respectively. Let $Y$ be a $3$-manifold with boundary, with 
$Y\hookrightarrow \partial X_1$ and $-Y\hookrightarrow \partial X_2$, such that $\partial Y$ has transverse intersections with $L_1$ and $L_2$. Then $\partial Y$ cuts $L_1$ into two tangles $T_2 = L_1 \cap Y$ and $-T_1 = L_1 \cap (\partial X_1 \setminus Y)$, and cuts $L_2$ into two tangles $-T_2' = L_2 \cap Y$ and $T_3 = L_2 \cap (\partial X_2 \setminus Y)$. 

Suppose that $T_2' = T_2$. Then we can glue $(X_1, -T_1 \cup T_2)$ and $(X_2, -T_2 \cup T_3)$ together along $(Y, T_2)$ to obtain $(X_1 \cup X_2, -T_1 \cup T_3)$. A lasagna filling of $(X_1, -T_1 \cup T_2)$ and a lasagna filling of $(X_2, -T_2 \cup T_3)$ glue to create a lasagna filling of $(X_1 \cup X_2, -T_1 \cup T_3)$ in a bilinear way, inducing a map
\begin{equation} \label{eq:naive_gluing_map}
    \textup{glue}_{T_2} \colon \mathcal{S}(X_1, -T_1 \cup T_2) \otimes_{\mathbb{Z}} \mathcal{S}(X_2, -T_2 \cup T_3) \to \mathcal{S}(X_1 \cup X_2, -T_1 \cup T_3).
\end{equation}

We are now in a position to define  categories associated with $3$-manifolds. The following definition was introduced and studied in \cite[Definition 4.5]{MWW1} in the case when the $3$-manifold is the $3$-ball.
Let $Y$ be an oriented, compact $3$-manifold with boundary, and let $\iota \colon \Sigma \to \partial Y$ be a (not necessarily orientation-preserving) diffeomorphism.

\begin{definition} \label{def:bordered-3-manifold-category}
To the $3$-manifold $Y$ and a finite collection of framed points $P \subset \Sigma$ we associate the \textit{$3$-manifold category} $\mathcal{S}(Y, P)$ defined as follows.
\begin{itemize}
\setlength{\itemsep}{5pt}
        \item[-] The objects of $\mathcal{S}(Y, P)$ are oriented framed tangles $T$ properly embedded in $Y$ such that $\partial T = \iota(P)$.
        \item[-] Given two tangles $T, T'$, let $L$ denote the oriented framed link $$(-T' \times \{-1\}) \cup  (\iota(P) \times I) \cup (T \times \{1\})$$ in the boundary of $Y \times I$, as shown in \autoref{fig:3-manifold-category}. Then the morphism set between the tangles $T, T'$ is the $\mathbb{Z}^2$-graded $\mathbb{Z}$-module
        $$
        \Hom_{\mathcal{S}(Y, P)}(T, T') = \mathcal{S}(Y \times I, L).
        $$
\end{itemize}
Composition is modeled on stacking lasagna fillings along the interval factor and is implemented by the gluing map in \autoref{eq:naive_gluing_map}. Note that when $Y$ is a closed $3$-manifold, we must have an empty point set on the boundary. We will denote $\mathcal{S}(Y, \emptyset)$ simply as $\mathcal{S}(Y)$.
\end{definition}

\begin{figure}[ht!]
    \centering
    \includegraphics[width=0.4\linewidth]{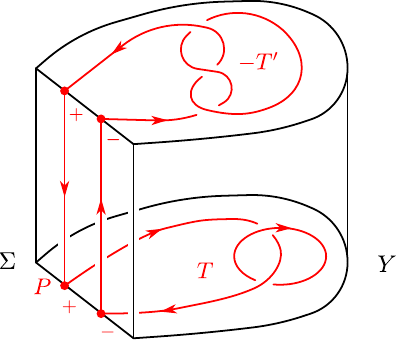}
    \caption{The $4$-manifold $Y \times I$ with the link $L = (-T' \times \{-1\}) \cup (P \times I) \cup (T \times \{1\})$ on its boundary.}
    \label{fig:3-manifold-category}
\end{figure}

\begin{remark} 
         Finite subsets $P\subset \Sigma$ and tangles $T\subset Y$ are considered as submanifolds, as opposed to equivalence classes up to isotopy.
\end{remark}
\begin{remark}
    If the diffeomorphism $\iota \colon \Sigma \to \partial Y$ has no ambiguity in context, we will identify $\Sigma$ and $\partial Y$, and write $\partial Y = \Sigma$ if $\iota$ is orientation-preserving, write $\partial Y = -\Sigma$ if $\iota$ is orientation-reversing.
\end{remark}

\section{The invariant of cornered \texorpdfstring{$4$}{4}-manifolds and gluing formulas} \label{sec: gluing}

In this section we define the invariant of cornered $4$-manifolds and present our gluing theorems. The reader may consult \cite[Section~1.5]{Wall} for the definition and basic properties of manifolds with corners. In \autoref{subsec:two} we consider the case of gluing two distinct, possibly cornered, oriented $4$-manifolds together. Then, motivated by the theory of trisections, in \autoref{subsec:self} we consider what happens when we glue a $4$-manifold to itself in a prescribed way. In \autoref{subsec:trisections}, we discuss an application of our results to the theory of trisections. 

We construct an invariant of oriented, compact $4$-manifolds $X$ with corners, $\partial X = (-Y_1) \cup_{\Sigma} Y_2$, where $Y_1$ and $Y_2$ are two oriented, compact $3$-manifolds with $\partial Y_1 = \partial Y_2 = \Sigma$ (note that we allow the possibility of $\Sigma$ to be empty, giving the case $\partial X = (-Y_1) \sqcup Y_2$). 
Suppose that $X$ is equipped with an orientation-preserving diffeomorphism $\phi \colon (-Y_1) \cup_{\Sigma} Y_2 \to \partial X$.
If we fix a point set $P \subset \Sigma$, then we obtain two categories $\mathcal{S}(Y_1, P)$ and $\mathcal{S}(Y_2, P)$ by \autoref{def:bordered-3-manifold-category}.

\begin{definition} \label{def:cornered-skein-lasagna-bimodule}
To the $4$-manifold $X$ with boundary as described above, we associate the \textit{cornered skein lasagna bimodule}, which is a bifunctor
$$F_{X,P} \colon \mathcal{S}(Y_1, P) \times \mathcal{S}(Y_2,P)^{op} \to \mathcal{V},$$ where $\mathcal{V}$ is the category of $\mathbb{Z}^2$-graded $\mathbb{Z}$-modules. For a pair of tangles $T_1 \in \Obj(\mathcal{S}(Y_1, P))$ and $T_2 \in \Obj(\mathcal{S}(Y_2,P))$, we define
\begin{equation} \label{eq:invariant}
   F_{X,P} (T_1,T_2) = \mathcal{S}(X, \phi(-T_1 \cup T_2)),
\end{equation}
where $\phi(-T_1 \cup T_2)$ is a link in $\partial X$, as shown in \autoref{fig:4-manifold-bimodule}.
For a pair of lasagna fillings $\alpha_1 \in \Hom_{\mathcal{S}(Y_1,P)}(T_1, T_1')$ and $\alpha_2 \in \Hom_{\mathcal{S}(Y_2,P)^{op}}(T_2, T_2')$, we define $F_{X,P}(\alpha_1, \alpha_2)$ to be the linear map induced by concatenating with $(\phi \times I)(\alpha_1 \cup \alpha_2)$. This definition extends bilinearly to linear combinations of lasagna fillings.
\end{definition}

\begin{figure}[ht!]
    \centering
    \includegraphics[width=0.4\linewidth]{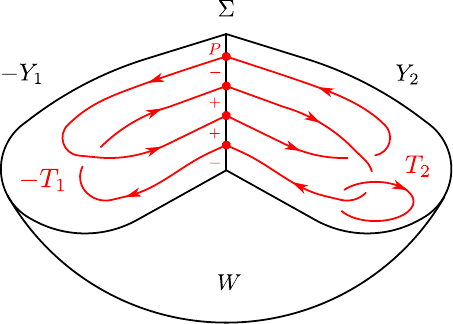}
    \caption{The $4$-manifold $X$ with the link $-T_1 \cup T_2$ on its boundary.}
    \label{fig:4-manifold-bimodule}
\end{figure}

Sometimes we will refer to the bifunctor $F_{X,P}$ as a $(\mathcal{S}(Y_1, P), \mathcal{S}(Y_2,P))$-bimodule.

\begin{remark}
    If the diffeomorphism $\phi \colon (-Y_1) \cup_{\Sigma} Y_2 \to \partial X$ has no ambiguity in context, we will identify $(-Y_1) \cup_{\Sigma} Y_2$ and $\partial X$, and write $\partial X = (-Y_1) \cup_{\Sigma} Y_2$.
\end{remark}

\begin{remark}
    Let $\alpha \in \Hom_{\mathcal{S}(Y_1,P)}(T_1, T_1')$ and $\gamma  \in \Hom_{\mathcal{S}(Y_2,P)^{op}}(T_2, T_2')$. For the linear map $F_{X,P}(\alpha, \gamma) \colon \mathcal{S}(X,\phi(-T_1 \cup T_2)) \to \mathcal{S}(X,\phi(-T_1' \cup T_2'))$ and an element $b \in \mathcal{S}(X,\phi(-T_1 \cup T_2))$, we will abuse notation and simply denote $(F_{X,P}(\alpha, \gamma))(b)$ as $\alpha \cdot b \cdot \gamma$. Furthermore, we will denote ${\id_{T_1}} \cdot b \cdot \gamma$ as $b \cdot \gamma$, and $\alpha \cdot b \cdot {\id_{T_2}}$ as $\alpha \cdot b$.
\end{remark}

\subsection{Gluing two pieces} \label{subsec:two}
Our goal now is to describe how the cornered skein lasagna bimodule behaves under gluing. Let $Y_1, Y_2, Y_3$ be three oriented compact $3$-manifolds with $\partial Y_1 = \partial Y_2 = \partial Y_3 = \Sigma$. Suppose that $X_1, X_2$ are oriented compact $4$-manifolds with $\partial X_1 = (-Y_1)\cup_\Sigma Y_2$ and $\partial X_2 = (-Y_2) \cup_\Sigma Y_3$. We will glue two such $4$-manifolds along $Y_2$ to get $X_1 \cup_{Y_2} X_2$, as depicted in \autoref{fig:gluing1}. 
Given two bimodules $F_{X_1, P} \colon \mathcal{S}(Y_1, P) \times \mathcal{S}(Y_2,P)^{op} \rightarrow \mathcal{V}$ and $F_{X_2, P} \colon \mathcal{S}(Y_2, P) \times \mathcal{S}(Y_3,P)^{op} \rightarrow \mathcal{V}$, we wish to define a tensor product between $F_{X_1, P}$ and $F_{X_2, P}$ to recover the bimodule $F_{X_1 \cup_{Y_2} X_2, P}$.

\begin{figure}[ht!]
\centering
    \begin{minipage}[c]{0.25\linewidth}   
    \includegraphics[width=\linewidth]{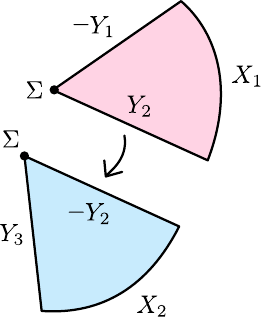}
    \end{minipage}
    \hspace{2em}
    \begin{minipage}[c]{0.5\linewidth}
    \includegraphics[width=\linewidth]{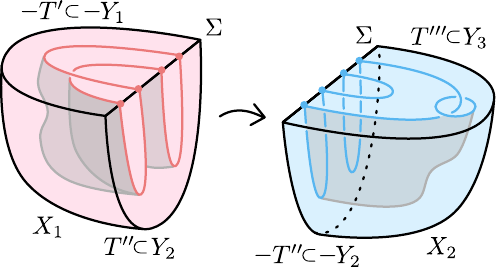}
    \end{minipage}
    \caption{Two schematics for gluing two cornered $4$-manifolds along part of their boundary. The left hand side is useful for a big picture view of what pieces are being glued, while the right hand side (in which all pieces have increased in dimension) is useful for seeing where the points, tangles, and surfaces (which will be relevant in \autoref{def:cornered-skein-lasagna-bimodule}) live. Here, the gray schematically represents lasagna fillings of $(X_1, T' \cup T'')$ and $(X_2,T'' \cup T''')$.
\label{fig:gluing1}}
\end{figure}

Fix an object $T'$ in $\mathcal{S}(Y_1, P)$ and an object $T'''$ in $\mathcal{S}(Y_3, P)$. Consider the equivalence relation $\sim_1$ on
$$
    \bigoplus_{T'' \in \Obj(\mathcal{S}(Y_2, P))} F_{X_1,P}(T', T'') \otimes_{\mathbb{Z}} F_{X_2, P}(T'', T''')
$$
generated by
\begin{equation} \label{eq: equiv rel1}
    (a \cdot \beta) \otimes c \sim_1 a \otimes (\beta \cdot c)
\end{equation}
where $a \in  \mathcal{S}(X_1, -T' \cup \widetilde{T}'')$, $c \in \mathcal{S}(X_2, -T'' \cup T''')$, and $\beta \in \Hom_{\mathcal{S}(Y_2,P)^{op}}(\widetilde{T}'',T'') = \Hom_{\mathcal{S}(Y_2,P)}(T'', \widetilde{T}'')$, for some objects $T'', \widetilde{T}''$ in $\mathcal{S}(Y_2, P)$.
See \autoref{fig:equiv1} for a schematic depiction of this equivalence relation. (Since we will have more equivalence relations in the next subsection, we number each of them for clarity.)

\begin{figure}[ht!]
\includegraphics[width=12cm]{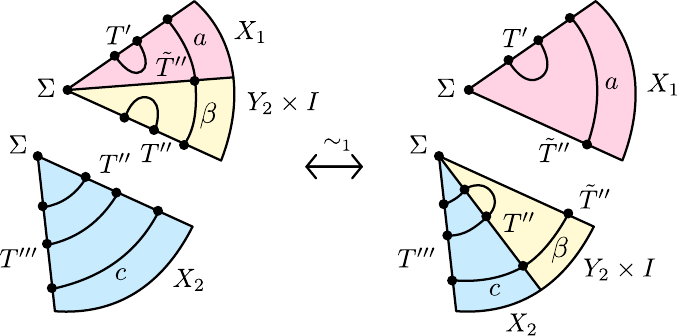}
\centering
\caption{A schematic representing the equivalence relation $(a \cdot \beta) \otimes c \sim_1 a \otimes (\beta \cdot c)$. In words, this relation captures the idea that acting by a morphism on the ``left'' side of $X_1$ or ``right'' side of $X_2$ gives the same result once the two pieces are glued along their common boundary. (Orientations have been omitted from this schematic for the sake of clarity.) \label{fig:equiv1}}
\end{figure}

\begin{definition}
    Given the $(\mathcal{S}(Y_1, P), \mathcal{S}(Y_2,P))$-bimodule $F_{X_1, P}$ and the $(\mathcal{S}(Y_2, P), \mathcal{S}(Y_3,P))$-bimodule $F_{X_2, P}$, we define their \textit{tensor product} over $\mathcal{S}(Y_2, P)$ to be the bifunctor 
    $$F_{X_1, P} \otimes_{\mathcal{S}(Y_2,P)} F_{X_2, P} \colon \mathcal{S}(Y_1, P) \times \mathcal{S}(Y_3,P)^{op} \rightarrow \mathcal{V},$$ 
    where $(F_{X_1, P} \otimes_{\mathcal{S}(Y_2,P)} F_{X_2, P})(T',T''')$ is defined to be 
    $$
    \bigoplus_{T'' \in \Obj
    (\mathcal{S}(Y_2, P))} F_{X_1,P}(T', T'') \otimes_{\mathbb{Z}} F_{X_2, P}(T'', T''') \bigg/ \sim_1.
    $$    
    On morphisms, we define
    $\alpha \cdot (a \otimes b) \cdot \beta = (\alpha \cdot a) \otimes (b \cdot \beta)$, where $\alpha \in \textup{Hom}(T', \widetilde{T}')$.
\end{definition}

Associativity will be discussed after \autoref{lem:lemma1}. 
Recall that by \autoref{eq:invariant}, we have
$$\bigoplus_{\substack{T'' \in \\ \Obj
    (\mathcal{S}(Y_2, P))}} F_{X_1,P}(T', T'') \otimes_{\mathbb{Z}} F_{X_2, P}(T'', T''') =  
\bigoplus_{\substack{\textup{$T''$: tangle in $Y_2$} \\ \partial T'' = P}} \mathcal{S}(X_1, -T' \cup T'') \otimes_{\mathbb{Z}} \mathcal{S}(X_2, -T'' \cup T''').$$
This motivates the following lemma.
    
\begin{lemma} \label{lem:lemma1}
    The map $$\tau \colon \left( \bigoplus_{T''\in \Obj
    (\mathcal{S}(Y_2, P))} \mathcal{S}(X_1, -T' \cup T'') \otimes_{\mathbb{Z}} \mathcal{S}(X_2, -T'' \cup T''') \bigg/ \sim_1 \right) \rightarrow \mathcal{S}(X_1 \cup X_2, -T' \cup T''')$$ sending $[a \otimes c]$ to $a \cup c$ is an isomorphism.
\end{lemma}

\begin{proof}
Consider the map $\widehat\tau$ which has the same domain as $\tau$ but without taking the quotient by $\sim_1$, and the same codomain as $\tau$. The map $\widehat\tau$ is defined by gluing the skein lasagna fillings,  $\widehat\tau(a\otimes c)=a\cup c$. It was observed in \cite[Proposition 2.3]{RW} that this map, which is the same as the map in \autoref{eq:naive_gluing_map} when the $4$-manifolds in question do not have corners, is surjective. Note that the values of $\widehat\tau$ on $(a \cdot \beta) \otimes c$ and on $a \otimes (\beta \cdot c)$ are both equal to the result of gluing the skein lasagna fillings $a\cup\beta\cup c$ in $X_1\cup (Y_2\times I)\cup X_2$. Therefore $\widehat\tau((a \cdot \beta) \otimes c - a \otimes (\beta \cdot c))=0$ and $\widehat\tau$ descends to a well-defined map $\tau$. 

To prove that it is an isomorphism, we will construct its inverse $\tau^{-1}$. Denote $X:=X_1\cup X_2$ and let $F$ be a skein lasagna filling in $\mathcal{S}(X, -T' \cup T''')$. Let $B$ denote the union of the input balls of $F$, $S$ the union of surfaces in $F$, and ${\mathfrak S}:=S\cup B$. After an isotopy we can assume that ${\mathfrak S}$ is transverse to $Y_2$. Here {\em transverse} means that $B\cap Y_2=\emptyset$, and the surface $S$ in the lasagna filling is transverse to $Y_2$ in the usual sense, as submanifolds of $X$. Transversality can be achieved because $B$ is a neighborhood in $X$ of a finite collection points which can be moved off of $Y_2$, and $B$ can be shrunk by an isotopy  to be disjoint from $Y_2$ as well. Then $(\partial S)\cap Y_2=\emptyset$ and $S\pitchfork Y_2$ is arranged as usual. 

Considering
$$T'':=S\cap Y_2,\; a:=F\cap X_1, \; c:= F\cap X_2, 
$$ 
we send $F\in \mathcal{S}(X, -T' \cup T''')$ to 
$$a\otimes c\in \left( \bigoplus_{\substack{\textup{$T''$}}} \mathcal{S}(X_1, -T' \cup T'') \otimes_{\mathbb{Z}} \mathcal{S}(X_2, -T'' \cup T''')  \right).$$
Here $a=F\cap X_1$ means that the input balls and the surfaces of $a$ are given by ${\mathfrak S}\cap X_1$, and the Khovanov-Rozansky homology labels in $a$ of the boundary links in $\partial B$ are inherited from~$F$. 

Next we show that the results $a\otimes c$, $a'\otimes c'$ of two isotopies of a skein lasagna filling $F$ are equal  modulo the equivalence relation $\sim_1$.
\begin{figure}[h]
\includegraphics[height=3.7cm]{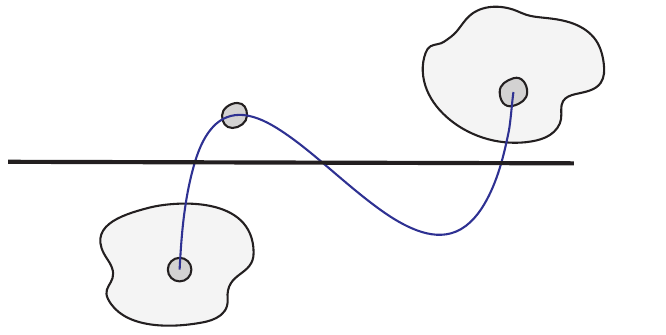}
{\small
\put(-197,61){$Y_2$}
\put(-130,12){$B$}
\put(-169,18){$D$}
\put(-52,35){$\gamma$}
}
\caption{An isotopy of an input ball.}
\label{fig:isotopy}
\end{figure}
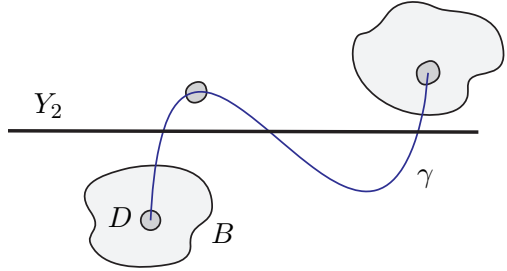
Consider two isotopic (in $X$) representatives of $F$, ${\mathfrak S}=S\cup B$ and ${\mathfrak S}'=S'\cup B'$. Pick a point in each input ball $B$; for simplicity of notation suppose $B$ consists of a single ball and denote the point by $p\in B$. 

Let $H$ denote the given ambient isotopy $H\colon X\times I\to X\times I$ taking $B$ to $B'$, and consider the inclusion ${\mathfrak S}\times I\hookrightarrow X\times I$. Consider the restriction of the isotopy $H$ to $\{ p\}\times I$; the path $\gamma$ given by  the track of this isotopy $\{ p\}\times I\to X\times I\to X$ is illustrated in Figure \ref{fig:isotopy}. We can assume that $\gamma$ is transverse to $Y_2$. Let $\{t_i\}$ be a finite collection of times when $\gamma$ intersects $Y_2$.

Consider a small neighborhood of $p$: a $4$-ball $D\subset B$. This neighborhood may be chosen sufficiently small so that near each time $t_i$ the isotopy $H$ moves the ball $D$ from $X_1$ straight across $Y_2$ to $X_2$, or vice versa. 
Consider an ambient isotopy that shrinks $B$ to $D$.
Then local isotopies that take $D$ across $Y_2$ give lasagna fillings that are equivalent under $\sim_1$. Note that when an input ball $B$ is pushed across $Y_2$, all of the surfaces $S$ attached to $B$ are isotoped as well, so the intersection $S\cap Y_2$ changes.

There are finitely many distinct times during the isotopy $H$ when the topology of the intersection of $\mathfrak{S}$ with $X_1$ and with $X_2$ changes. These {\em singular events} consist of the balls $B$ being pushed across $Y_2$ as above, one at a time, and the intersection $S\cap Y_2$ undergoing Morse singularities (away from the boundaries of $S$). 
The final step in the proof is to use fragmentation of $H$ into a finite sequence of isotopies so that each of the finitely many singular events described above is implemented by its own local isotopy. 

The isotopy $H$ has a corresponding vector field $V$ on $X\times I$, $V(x,t)=\frac{\partial}{{\partial }t}H(x,t)$ whose vertical component is $\partial/\partial t$. Rescaling the horizontal component of the vector field by a function $g\colon X\times I \to {\mathbb R}_{\geq 0}$, $V(x,t)\mapsto g(x,t)V(x,t)$ has the effect of reparametrizing the integral curves of the vector field. A singular event takes place in $U \colon ={\mathcal B}\times [t_i-\delta, t_i+\delta]\subset X\times I$ where ${\mathcal B}\subset X$ is a $4$-ball neighborhood of a point in $Y_2$, and $\delta>0$ is small.  
It follows from the fragmentation lemma \cite{Banyaga} that the isotopy $H$ can be represented as a composition of local isotopies $H_i$ supported in neighborhoods as above.
Since $H_i$ is the identity on the complement of $U$, the difference between the tensor products $a\otimes c$ before and after a given singular event differ by $\sim_1$.

Skein lasagna relations involve a collection of input balls engulfed in a single $4$-ball. Since all of the input balls of a given lasagna filling in $X$ may be isotoped into one of the pieces, say $X_1$, the relation in ${\mathcal S}(X)$ is implied by relations in ${\mathcal S}(X_1)$.
\end{proof}

The following theorem summarizes the discussion above, stating a concise tensor product gluing formula.

\begin{theorem} \label{thm:theorem1}
    Let $\Sigma$ be an oriented, closed surface containing a finite set of signed points $P$. Let $Y_1, Y_2, Y_3$ be three oriented, compact $3$-manifolds with $\partial Y_1 = \partial Y_2 = \partial Y_3 = \Sigma$, and $X_1, X_2$ be oriented, compact $4$-manifolds with $\partial X_1 = (-Y_1)\cup_\Sigma Y_2$ and $\partial X_2 = (-Y_2) \cup_\Sigma Y_3$. Then
    \begin{equation} \label{eq: nat iso}
        F_{X_1, P} \otimes_{\mathcal{S}(Y_2,P)} F_{X_2, P} \cong F_{X_1 \cup X_2, P}
    \end{equation}
    as $(\mathcal{S}(Y_1, P), \mathcal{S}(Y_3, P))$-bimodules.
\end{theorem}

\begin{proof}
    We will construct the natural isomorphism between the bifunctors in \autoref{eq: nat iso}.
    By definition, $$F_{X_1 \cup X_2, P}(T',T''') = \mathcal{S}(X_1 \cup X_2, -T' \cup T''')$$
    for any $T' \in \textup{Obj}(\mathcal{S}(Y_1,P))$ and $T''' \in \textup{Obj}(\mathcal{S}(Y_3,P))$. By \autoref{lem:lemma1}, the map
    $$\tau \colon (F_{X_1, P} \otimes_{\mathcal{S}(Y_2,P)} F_{X_2, P})(T',T''') \rightarrow F_{X_1 \cup X_2}(T',T''')$$ sending $a \otimes c$ to $a \cup c$ is an isomorphism. Here the map $\tau$ depends on the objects $T'$ and $T'''$, so we will denote it $\tau_{T', T'''}$.

    To prove the naturality of $\{\tau_{T', T'''}\}$, we need to check that the following diagram is commutative:
    \[\begin{tikzcd}
    	{F_{X_1, P} \otimes F_{X_2, P}(T', T''')} & {F_{X_1 \cup X_2}(T', T''')} \\
    	{F_{X_1, P} \otimes F_{X_2, P}(\widetilde{T}', \widetilde{T}''')} & {F_{X_1 \cup X_2}(\widetilde{T}', \widetilde{T}''')}
    	\arrow["\tau_{T', T'''}", from=1-1, to=1-2]
    	\arrow["{\alpha \cdot (-) \cdot \gamma}"', from=1-1, to=2-1]
    	\arrow["{\alpha \cdot (-) \cdot \gamma}", from=1-2, to=2-2]
    	\arrow["{\tau_{\widetilde{T}',\widetilde{T}'''}}"', from=2-1, to=2-2]
    \end{tikzcd}\]
    where $\widetilde{T}' \in \textup{Obj}(\mathcal{S}(Y_1,P))$, $\widetilde{T}''' \in \textup{Obj}(\mathcal{S}(Y_3,P))$, and $\alpha$ is a lasagna filling in $\textup{Hom}_{\mathcal{S}(Y_1,P)}(T', \widetilde{T}')$, $\gamma$ is a lasagna filling in $\textup{Hom}_{\mathcal{S}(Y_3,P)^{op}}(T''', \widetilde{T}''')$.
    
    Indeed, for any $a \otimes c \in F_{X_1, P} \otimes F_{X_2, P}(T', T''')$, we have $\tau_{\widetilde{T}',\widetilde{T}'''}(\alpha \cdot (a \otimes c) \cdot \gamma) = \tau_{\widetilde{T}',\widetilde{T}'''}((\alpha \cdot a) \otimes (c \cdot \gamma)) = \tau_{\widetilde{T}',\widetilde{T}'''}((\alpha \cup a) \otimes (c \cup \gamma)) = \alpha \cup a \cup c \cup \gamma$, and $\alpha \cdot (\tau_{T', T'''}(a \otimes c)) \cdot \gamma = \alpha \cdot (a \cup c) \cdot \gamma = \alpha \cup a \cup c \cup \gamma$.
\end{proof}

\begin{remark}
Considering the special case where $\Sigma=\emptyset$,  \autoref{thm:theorem1} gives a gluing formula for $4$-manifolds with boundary (without corners). 
Some important special cases of such gluing have been previously considered in the literature, including for example the case of attaching $2$-handles to the $4$-ball \cite{MN} and more general handle decompositions \cite{MWW1}.
\end{remark}

\begin{remark}
It follows from \autoref{thm:theorem1} that the tensor product of bimodules associated with cornered $4$-manifolds is associative.
\end{remark}

\subsection{Self-gluing} \label{subsec:self}
Our next two gluing formulas, \autoref{thm:theorem2} and \autoref{thm:theorem3}, were originally motivated by the theory of trisections; see \autoref{subsec:trisections} for more details. However we present our results here in a more general setting. 

Given an oriented, compact $3$-manifold $Y$ with $\partial Y = \Sigma$ and a fixed point set $P \subset \Sigma$, we obtain a category $\mathcal{S}(Y, P)$ by \autoref{def:bordered-3-manifold-category}. Suppose that $X$ is an oriented, compact $4$-manifold with an explicit diffeomorphism $\phi \colon (-Y) \cup_{\Sigma} Y \to \partial X$. Then \autoref{def:cornered-skein-lasagna-bimodule} gives us the cornered skein lasagna $(\mathcal{S}(Y, P), \mathcal{S}(Y, P))$-bimodule
$$
    F_{X, P} \colon \mathcal{S}(Y,P) \times \mathcal{S}(Y,P)^{op} \to \mathcal{V}
$$
sending $(T',T'')$ to $\mathcal{S}(X, \phi(-T' \cup  T''))$.
We glue $X$ to itself along $Y$, and define $\bar{X} := X\big/ \{\phi(y) \sim \phi(-y)\}$ for $y \in Y$. See \autoref{fig:gluing2}. Our next goal is to study the effect of this self-gluing on the bimodule.

\begin{figure}[ht!]
\includegraphics[width=9cm]{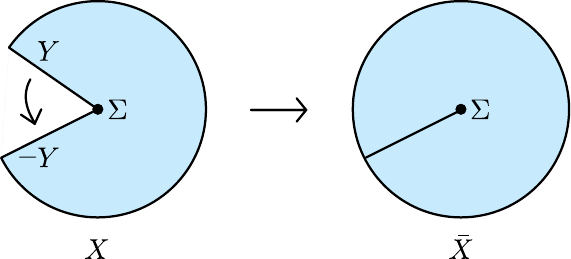}
\centering
\caption{A schematic for self-gluing a cornered $4$-manifold along its boundary. \label{fig:gluing2}}
\end{figure}

Consider the equivalence relation $\sim_2$ on
$$
    \bigoplus_{T \in \Obj(\mathcal{S}(Y, P))} F_{X_,P}(T, T)
$$
generated by
$$
    a \cdot \beta \sim_2 \beta \cdot a
$$
where $a \in F_{X,P}(T, \widetilde{T})$ and $\beta \in \Hom_{\mathcal{S}(Y,P)^{op}}(\widetilde{T},T) = \Hom_{\mathcal{S}(Y,P)}(T,\widetilde{T})$ for some object $\widetilde{T}$ in $\mathcal{S}(Y,P)$. See \autoref{fig:equiv2} for a schematic depiction of this equivalence relation.

\begin{figure}[ht!]
\hspace*{-.5in}
\includegraphics[width=11cm]{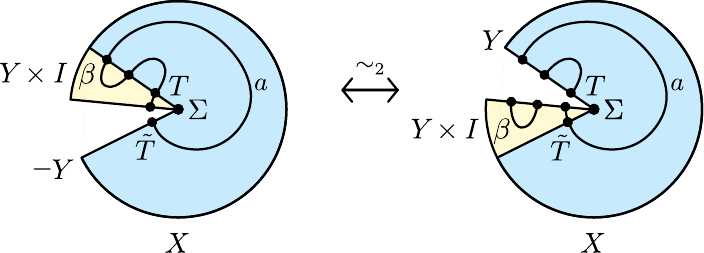}
\centering
\caption{A schematic representing the equivalence relation $a \cdot \beta \sim_2 \beta \cdot a$. In words, this relation captures the idea that acting by a morphism on the ``left'' or ``right'' side of $X$ gives the same result once the manifold is self-glued along its boundary. \label{fig:equiv2}}
\end{figure}

\begin{definition} \label{def:self}
    Given the $(\mathcal{S}(Y, P), \mathcal{S}(Y, P))$-bimodule $F_{X,P}$, we define its \textit{zero\textsuperscript{th} Hochschild homology} to be the $\mathbb{Z}^2$-graded $\mathbb{Z}$-module
    $$
        \textup{HH}_0(F_{X, P}) = \bigoplus_{T \in \Obj(\mathcal{S}(Y, P))} F_{X_,P}(T, T) \bigg/ \sim_2.
    $$
\end{definition}

Our process for self-gluing will now proceed in several steps. We will remove a neighborhood $\nu(\Sigma)$ of the surface $\Sigma$ first, then self-glue, and then fill $\Sigma$ back in. Without the added step of removing and gluing back in a neighborhood of $\Sigma$, we would be forced to require lasagna fillings to intersect $\Sigma$ in a point set $P$.
When we remove a neighborhood of $\Sigma$ with a given point set $P \subset \Sigma$, the points on $\Sigma$ will be exchanged for $|P|$ meridional circles on the boundary. For convenience we will identify $P \times S^1 \subset \partial (\Sigma \times D^2)$ and its image under a diffeomorphism $\Sigma \times D^2\to \nu(\Sigma)$.

Let $q \colon X \rightarrow \bar{X}$ be the quotient map given by self-gluing along $\phi$. Let $X^0 \vcentcolon = X \setminus (\Sigma \times D^2_{\angle})$ and $\bar{X}^0 \vcentcolon = \bar{X} \setminus (\Sigma \times D^2)$, where $D^2_{\angle} \subset X$ is the preimage of the fiber $D^2$ of the normal bundle $\Sigma \times D^2$ of $\Sigma\subset  \bar{X}$. Note that $\partial \bar{X}^0 \cong \Sigma \times S^1$. See \autoref{fig:notation} for a summary of the notation for our $4$-manifold at various stages of removing and gluing.

\begin{figure}[ht!]
\includegraphics[width=15cm]{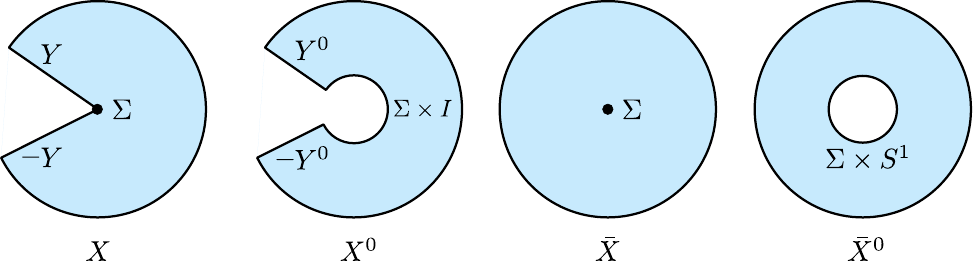}
\centering
\caption{A schematic representing $X$, $X^0$, $\bar{X}$, and $\bar{X}^0$. Here, a ``bar'' means that the manifold has been glued along $Y$, and a ``zero'' means that a neighborhood of the surface $\Sigma$ has been removed. By $Y^0$, we mean $Y\setminus ((\Sigma \times D^2_{\angle})\cap Y)$.
\label{fig:notation}}
\end{figure}

Recall from \autoref{eq:invariant} that 
$F_{X,P} (T,T) = \mathcal{S}(X, \phi(-T \cup T))$. Replacing $X$ with $X^0$ amounts to removing a regular neighborhood $\nu(\Sigma)$ of $\Sigma$ in $X$. The inclusion $X^0\subset X$ induces an isomorphism of the skein lasagna modules, as follows. The boundary link $\phi(-T \cup T)$ is modified near each point $x\in P\subset \Sigma$ by replacing the two radial segments of $\{ x\}\times \partial D^2_{\angle}$ with $\{ x\}$ times the circular arc in $\partial D^2_{\angle}$. Denote the modified link by $\phi(-T' \cup A \cup T')$, where $A$ are the circular arcs. The boundary of any skein lasagna filling $F$ in $\mathcal{S}(X, \phi(-T \cup T))$ intersects $\Sigma$ in $P$, and the relations in the skein lasagna module may be assumed to miss $\nu(\Sigma)$. Removing $F\cap (\Sigma\times D^2_{\angle})$ gives a skein lasagna filling  $F^0$ in $\mathcal{S}(X^0, \phi(-T' \cup A \cup T'))$.

\begin{lemma} \label{lem:lemma2}
    In the notation as above, let $F^0_{T'}$ denote the skein lasagna filling in $\mathcal{S}(\bar{X}^0,P \times S^1)$ obtained by identifying the two copies of $T'$. The map
    $$
    \rho \colon \bigg( \bigoplus_{T \in \Obj(\mathcal{S}(Y, P))} F_{X_,P}(T, T) \bigg/ \sim_2  \bigg) \to  \mathcal{S}(\bar{X}^0,P \times S^1)
    $$
    sending $F$ to $F^0_{T'}$ is an isomorphism.
\end{lemma}

\begin{proof}
The construction of the inverse map is directly analogous to the proof of \autoref {lem:lemma1}.
\end{proof}

The following theorem gives a concise statement of the self-gluing formula.

\begin{theorem} \label{thm:theorem2}
    Let $\Sigma$ be an oriented, closed surface containing a finite set of signed points $P$. Let $Y$ be an oriented, compact $3$-manifold with $\partial Y = \Sigma$, and $X$ be an oriented, compact $4$-manifold with $\partial X = (-Y) \cup_{\Sigma} Y$. Then
    $$
    \textup{HH}_0(F_{X, P})\cong \mathcal{S}(\bar{X}^0,P \times S^1)
    $$
    as $\mathbb{Z}^2$-graded $\mathbb{Z}$-modules.
\end{theorem}

\begin{proof}
    By definition, $$ \textup{HH}_0(F_{X, P}) =   \bigoplus_{T \in \Obj(\mathcal{S}(Y, P))} F_{X_,P}(T, T) \bigg/ \sim_2 .$$
    Thus we have a map
    $$\textup{HH}_0( F_{X, P} ) \rightarrow \mathcal{S}(\bar{X}^0,P \times S^1)$$
    sending $F$ to $F^0_{T'}$, which is an isomorphism by \autoref{lem:lemma2} as desired.
\end{proof}

Finally, we glue back in the neighborhood $\Sigma \times D^2$ -- in other words, we go from $\bar{X}^0$ to $\bar{X}$ -- and describe the result on our invariant. This will require the definition of a third equivalence relation, which will be generated by two equivalencies.

Let $P, P'$ be two finite sets of signed points in $\Sigma$ of the same cardinality.  Let $b$ be an oriented braid properly embedded in $\Sigma \times I$ from $P$ to $P'$, meaning $\partial b = (-P) \times \{-1\} \cup P' \times \{1\}$. Then $b \times S^1$ is an oriented surface in $\Sigma \times S^1 \times I$, with boundary $(-P \times S^1) \times \{-1\} \cup (P' \times S^1) \times \{1\}$. By concatenating with the lasagna fillings of $(\bar{X}^0, P \times S^1)$, we obtain a linear map
$$
\beta(b) \colon \mathcal{S}(\bar{X}^0, P \times S^1) \to \mathcal{S}(\bar{X}^0, P' \times S^1).
$$

On the other hand, suppose $p_+, p_- \in \Sigma \setminus P$ are a pair of points with opposite signs. Then $\{p_+, p_-\} \times S^1$ bounds an annulus $A_{p_+, p_-}$ in $\bar{X}^0$. For any lasagna filling $F$ of $(\bar{X}^0, P \times S^1)$, we can take the connected sum $F \# A_{p_+, p_-}$ between any component of $F$ and $A_{p+,p_-}$. Note that $F \# A_{p_+, p_-}$ is a lasagna filling of $(\bar{X}^0, (P \cup \{p_+, p_-\}) \times S^1)$. 

Now consider the equivalence relation $\sim_3$ on
$$
\bigoplus_P \mathcal{S}(\bar{X}^0,P \times S^1)
$$
generated by
$$
\beta(b)(F) \sim_3 F, \quad F \# A_{p_+, p_-} \sim_3 F
$$
for any $F \in \mathcal{S}(\bar{X}^0,P \times S^1)$, any braid $b$ in $\Sigma \times I$ from $P$ to some $P' \sim P$,  and any pair of points $p_+, p_- \in \Sigma \setminus P$ with opposite signs. See \autoref{fig:equiv3} for a schematic depiction of this equivalence relation.

\begin{figure}[ht!]
\includegraphics[width=12cm]{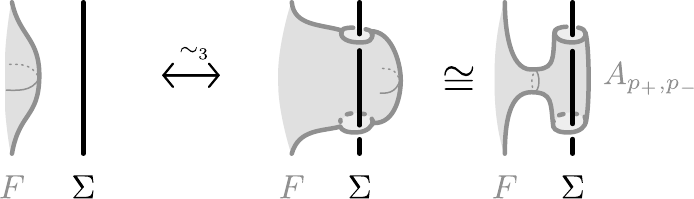}
\centering
\caption{A schematic representing the equivalence relation $ F \# A_{p_+, p_-} \sim_3 F$. In words, the relation accounts for the possibility of a lasagna filling $F$ being isotoped to intersect the central surface $\Sigma$ in a pair of points with opposite sign. This isotopy can be thought of as performing a connected sum between the lasagna filling $F$ and the annulus $A_{p_+, p_-}$.
\label{fig:equiv3}}
\end{figure}

\begin{theorem} \label{thm:theorem3}
   There is an isomorphism
    \[\big(\bigoplus_P \mathcal{S}(\bar{X}^0,P \times S^1) \big/ \sim_3\big) \cong \mathcal{S}(\bar{X}).\]
\end{theorem}

\begin{proof}
    Consider the map   $\phi\colon\bigoplus_P \mathcal{S}(\bar{X}^0,P \times S^1) \to \mathcal{S}(\bar{X})$ sending a lasagna filling $F$ with boundary $P\times S^1\subset \Sigma\times S^1$ to a lasagna filling in $\bar{X}$ obtained from $F$ by capping it off with meridional disks $P\times D^2$. 
   This map is surjective: starting with any lasagna filling $F'$ in $\mathcal{S}(\bar{X})$, make it transverse to $\Sigma\times\{0\}\subset \Sigma\times D^2\subset \bar{X}$ so the input balls are disjoint from $\Sigma$, and the surfaces in $F$ intersect $\Sigma$ in a finite collection of points $P$. Intersecting with the complement of an open tubular neighborhood of $\Sigma$, we obtain a filling $F$ in $\mathcal{S}(\bar{X}^0,P \times S^1)$ which is sent to $F'$ by $\phi$. An argument analogous to the proof of \cite[Theorem~1.1]{MN} shows that isotopic representatives of $F'$ are sent to elements related by $\sim_3$. 
\end{proof}

\begin{remark} The effect on the skein lasagna module of filling in $\Sigma\times D^2$ to obtain $\bar{X}$ from $\bar{X}^0$ can also be understood in terms of attaching $2, 3, 4$-handles as in \cite{MWW1}.
\end{remark}

\subsection{Application to trisections} \label{subsec:trisections}

Trisections, due to Gay and Kirby \cite{GK}, are decompositions of $4$-manifolds which are analogous to Heegaard splittings of $3$-manifolds.

\begin{definition}[\cite{GK}]
    A \textit{$(g;k_1,k_2,k_3)$-trisection} of a closed, connected, oriented, smooth $4$-manifold $X^4$ is a decomposition $X^4 = X_1 \cup X_2 \cup X_3$ satisfying:
    \begin{enumerate}
        \item each $X_i$ is diffeomorphic to $\natural_{k_i} (S^1 \times B^3)$, that is, a $4$-dimensional $1$-handlebody $Z_{k_i}$,
        \item each $Y_{i, i+1}:=X_i \cap X_{i+1}$ (taking $i \pmod 3$) is diffeomorphic to $\natural_g (S^1 \times B^2)$, that is, a genus $g$ $3$-dimensional handlebody $H_g$, and
        \item the triple intersection $X_1 \cap X_2 \cap X_3$ is a closed, oriented, genus $g$ surface $\Sigma_g$, which is referred to as the \textit{central surface}.
    \end{enumerate}
This decomposition is depicted schematically in \autoref{fig:trisection} (left).
\end{definition}

\begin{figure}[ht!]
\centering
    \begin{minipage}[c]{0.25\linewidth}   
    \includegraphics[width=\linewidth]{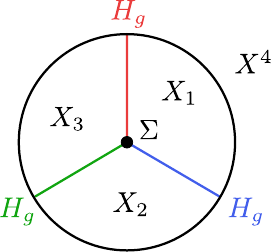}
    \end{minipage}
    \hspace{2em}
    \begin{minipage}[c]{0.25\linewidth}
    \includegraphics[width=\linewidth]{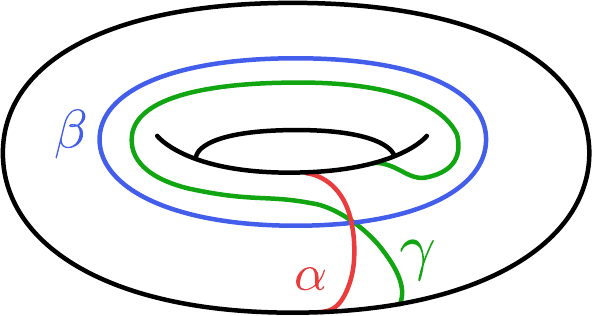}
    \end{minipage}
    \caption{
    On the left is a schematic depiction of a trisection $X^4 = X_1 \cup X_2 \cup X_3$. On the right is a trisection diagram for $\mathbb{CP}^2$.
\label{fig:trisection}}
\end{figure}

\autoref{thm:theorem1} allows us to glue two trisection pieces together along a genus $g$ handlebody, so in total we can glue three pieces to get a $4$-manifold $X'$ with boundary $\partial X \cong -H_g \cup_{\Sigma_g} H_g$, where $H_g$ is a genus $g$ handlebody. The step that remains, in order to build the closed, trisected $4$-manifold $X$, is to glue $X'$ to itself by identifying $-H_g$ and $H_g$. We can accomplish this with \autoref{thm:theorem2} and \autoref{thm:theorem3}.

However, to define the categories corresponding to the $3$-manifolds $Y_{i, i+1}$ (\autoref{def:bordered-3-manifold-category}), we need to specify the diffeomorphisms $\iota_{i,i+1} \colon \Sigma_g \to H_g$. Similarly, to obtain the bimodules corresponding to the cornered $4$-manifolds $X_i$ (\autoref{def:cornered-skein-lasagna-bimodule}), we need to determine the diffeomorphisms $\phi_i \colon -H_g \cup_{\Sigma_g} H_g \to Z_{k_i}$. Next we will describe how to extract the data of these diffeomorphisms from trisection diagrams.

\begin{definition}
    A \textit{$(g;k_1,k_2,k_3)$-trisection diagram} is a tuple $(\Sigma_g;\alpha,\beta,\gamma)$ such that:
    \begin{itemize}
    \setlength{\itemsep}{5pt}
        \item[-] each of $\alpha, \beta$, and $\gamma$ is a cut system of curves for $\Sigma_g$, and
        \item[-] each of $(\Sigma_g;\alpha,\beta), (\Sigma_g;\beta,\gamma)$, and $(\Sigma;\gamma,\alpha)$ is a genus $g$ Heegaard diagram for $\#^{k_i}S^1 \times S^2$ (respectively).
    \end{itemize}
See \autoref{fig:trisection} (right) for an example. We refer the readers to \cite{GK} for a detailed discussion of the relationship between trisections and trisection diagrams.
\end{definition}

Given a trisection diagram $(\Sigma_g;\alpha,\beta,\gamma)$ of $X$, consider the diffeomorphism $\iota_{3,1} \colon \Sigma_g \to \partial H_g$ that sends the $g$ $\alpha$-curves on $\Sigma_g$ to the $g$ meridian curves on $\partial H_g$. Fix a finite set of signed points $P$ on $\Sigma_g$. By  \autoref{def:bordered-3-manifold-category}, the diffeomorphism $\iota_{3,1}$ gives rise to a category $\mathcal{S}(H_g,P)$. Similarly, we can consider the diffeomorphism $\iota_{1,2} \colon \Sigma_g \to \partial H_g$ sending the $g$ $\beta$-curves on $\Sigma_g$ to the $g$ meridian curves on $\partial H_g$, and the diffeomorphism $\iota_{2,3} \colon \Sigma_g \to \partial H_g$ sending the $g$ $\gamma$-curves on $\Sigma_g$ to the $g$ meridian curves on $\partial H_g$. To distinguish the three resulting categories, we denote the category corresponding to the diffeomorphism $\iota_{i,j}$ by $\mathcal{S}(Y_{i,j}, P)$.

We know that the Heegaard diagram $(\Sigma; \alpha, \beta)$ is equivalent to the standard Heegaard splitting of $\partial Z_{k_1}$. Let $\phi_1 \colon -H_g \cup_{\Sigma_g} H_g \to \partial Z_{k_1}$ be the diffeomorphism induced by this equivalence. By definition, the diffeomorphism $\phi_1$ gives rise to an $(\mathcal{S}(Y_{3,1}, P), \mathcal{S}(Y_{1,2}, P))$-bimodule $F_{Z_{k_1}, P}$, which will be denoted as $F_{X_1,P}$. We similarly consider the Heegaard diagrams $(\Sigma_g; \beta, \gamma)$ and $(\Sigma_g; \gamma, \alpha)$, to obtain an  $(\mathcal{S}(Y_{1,2}, P), \mathcal{S}(Y_{2,3}, P))$-bimodule $F_{X_2, P}$, and an $(\mathcal{S}(Y_{2,3}, P), \mathcal{S}(Y_{3,1}, P))$-bimodule $F_{X_3, P}$. 

\begin{corollary}
    Let $X$ be a closed $4$-manifold. Given a trisection $X = X_1 \cup X_2 \cup X_3$, and a finite set of signed points $P$ in the central surface $\Sigma_g$, then
    $$
    \mathcal{S}(X \setminus \nu(\Sigma_g); P \times S^1) \cong \textup{HH}_0(F_{X_1, P} \otimes_{S(Y_{1,2};P)} F_{X_2, P} \otimes_{S(Y_{2,3};P)} F_{X_3, P}).
    $$
\end{corollary}

Here $\textup{HH}_0$ denotes the zero\textsuperscript{th} Hochschild homology; see \autoref{def:self}.

\section{The invariant of closed surfaces} \label{sec:other}

In this section we describe an extension of the theory down to two dimensions.

\begin{definition}
    To an oriented surface $\Sigma$ we associate the bicategory $\mathcal{S}(\Sigma)$ defined as follows.
    
    \begin{itemize}
    \setlength{\itemsep}{5pt}
        \item[-] The objects of $\mathcal{S}(\Sigma)$ are oriented closed $0$-manifolds embedded in $\Sigma$, that is, finite sets of signed points in $\Sigma$.

        \item[-] The $1$-cells between two sets of signed points $P$ and $Q$ are oriented tangles embedded in $\Sigma \times I$, with boundary $-Q \times \{-1\} \sqcup P \times \{1\}$. The identity $1$-cell from $P$ to $P$ is given by the trivial tangle $P \times I$.

        \item[-] Given tangles $b'$ and $b$ from $P$ to $Q$, let $L$ denote the oriented link
        $$
        (P \times \{1\} \times I) \cup (-b \times \{-1\}) \cup (-Q \times \{-1\} \times I) \cup (-b \times \{1\})
        $$
        Then the $2$-cells from $b'$ to $b$ are the elements in $\mathcal{S}(\Sigma \times I \times I, L)$.
    \end{itemize}

    Compositions of $1$-cells are stacking of tangles, followed by a rescaling of $I$. Horizontal composition of $2$-cells is modeled on stacking lasagna fillings along the first interval factor $I$, while vertical composition of $2$-cells is modeled on stacking lasagna fillings along the second interval factor $I$. Both are implemented by the gluing map in \autoref{eq:naive_gluing_map}, followed by a rescaling of $I$.

    Given $1$-cells $b_1, b_2, b_3$, the associator from $(b_1 \circ b_2) \circ b_3$ to $b_1 \circ (b_2 \circ b_3)$ is the isotopy between the two tangles given by rescaling $I$, considered as a lasagna filling embedded in $\Sigma \times I \times I$. Given a $1$-cell $b$ from $P$ to $Q$, the left unitor from $(Q \times I) \cup b$ to $b$ and the right unitor from $b \cup (P \times I)$ to $b$ are likewise given by isotopies.
\end{definition}

\begin{figure}[ht!]
    \centering
    \includegraphics[width=0.35\linewidth]{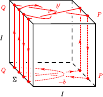}
    \caption{The $4$-manifold $\Sigma \times I \times I$ with the link $L = (P \times \{1\} \times I) \cup (-b \times \{-1\}) \cup (-Q \times \{-1\} \times I) \cup (b' \times \{1\})$ on its boundary.
    \label{fig:surface-double-category}}
\end{figure}

Let $\mathfrak{Cat}$ denote the bicategory of bimodules over small categories enriched in $\mathbb{Z}^2$-graded $\mathbb{Z}$-modules. The objects of $\mathfrak{Cat}$ are enriched small categories, the $1$-cells between two categories $\mathcal{C}$ and $\mathcal{D}$ are $(\mathcal{C}, \mathcal{D})$-bimodules, i.e. bifunctors from $\mathcal{C} \times \mathcal{D}^{op}$ to $\mathcal{V}$ (see \autoref{def:cornered-skein-lasagna-bimodule}). The identity $1$-cell $id_{\mathcal{C}}$ from $\mathcal{C}$ to $\mathcal{C}$ is given by the hom-functor sending $(x, y)$ to $\textup{Hom}(x, y)$. Given two $(\mathcal{C}, \mathcal{D})$-bimodules $M$ and $M'$, the $2$-cells from $M$ to $M'$ are natural transformations from $M$ to $M'$, sometimes called bimodule homomorphisms. (See also \cite[Chapter 7]{Handbook} for a related discussion of the bicategory of profunctors, and see \cite[Example 2.1.26]{JY} for a related example.)

Composition of $1$-cells is the tensor product of bimodules. Vertical composition of $2$-cells is composition of bimodule homomorphisms, and horizontal composition of $2$-cells is tensor product of bimodule homomorphisms. Let $M$ be a $(\mathcal{B}, \mathcal{C})$-bimodule, $N$ be a $(\mathcal{C}, \mathcal{D})$-bimodule, and $P$ be a $(\mathcal{D}, \mathcal{E})$-bimodule. The associator is given by the natural isomorphism $(M \otimes_{\mathcal{C}} N) \otimes_{\mathcal{D}} P \cong M \otimes_{\mathcal{C}} (N \otimes_{\mathcal{D}} P)$, and the left unitor and the right unitor are given by the natural isomorphisms $N \otimes_{\mathcal{D}} id_{\mathcal{D}} \cong N$ and $id_\mathcal{C} \otimes_{\mathcal{C}} N \cong N$ respectively.

\begin{definition}
    Let $\mathfrak{C}$ be a bicategory. A \textit{left $\mathfrak{C}$-module} is a pseudofunctor from $\mathfrak{C}$ to $\mathfrak{Cat}$. A \textit{right $\mathfrak{C}$-module} is a pseudofunctor from $\mathfrak{C}^{op}$ to $\mathfrak{Cat}$.
\end{definition}

To a $3$-manifold $Y$ with an orientation-preserving diffeomorphism $\iota \colon \Sigma \to \partial Y$ we associate a right module over the bicategory $\mathcal{S}(\Sigma)$.

\begin{definition}\label{def:3manifold module}
We define the right $\mathcal{S}(\Sigma)$-module $\mathcal{S}_Y \colon \mathcal{S}(\Sigma)^{op} \to \mathfrak{Cat}$ as follows.
\begin{itemize}
\setlength{\itemsep}{5pt}
    \item[-] For a signed point set $P \subset \partial Y$, define $\mathcal{S}_Y(P) = \mathcal{S}(Y,P)$, the category associated to the $3$-manifold $Y$ and the point set $P$ in \autoref{def:bordered-3-manifold-category}.
    
    \item[-] For a $1$-cell $b$ in $\mathcal{S}(\Sigma)^{op}$ from $Q$ to $P$, the $(\mathcal{S}(Y, Q), \mathcal{S}(Y,P))$-bimodule $\mathcal{S}_Y(b)$ is a bifunctor
    $$
    \mathcal{S}_Y(b) \colon \mathcal{S}(Y, Q) \times \mathcal{S}(Y, P)^{op} \to \mathcal{V}.
    $$
    For a pair of tangles $u \in \textup{Obj}(\mathcal{S}(Y, Q))$ and $u' \in \textup{Obj}(\mathcal{S}(Y, P))$, define
    $$
    \mathcal{S}_Y(b)(T, \widetilde{T}) = \mathcal{S}(Y \times I, (u \times \{-1\}) \cup b \cup (-u' \times \{1\})).
    $$
    See \autoref{fig:3-manifold-bimodule} (left). For a pair of lasagna fillings $\alpha, \alpha'$, define $\mathcal{S}_Y(b)(\alpha, \alpha')$ to be the linear map induced by concatenating $-\alpha$ and $-\alpha'$.
    
    \item[-] For a $2$-cell $\beta$ in $\mathcal{S}(\Sigma)^{op}$ from $b$ to $b'$, which is a lasagna filling in $\mathcal{S}(\Sigma \times I \times I, (P \times \{1\} \times I) \cup (-b \times \{-1\}) \cup (-Q \times \{-1\} \times I) \cup (b' \times \{1\})$ as shown in \autoref{fig:surface-double-category}, we first collapse both $\Sigma \times \{-1\} \times I$ and $\Sigma \times \{1\} \times I$ along the interval $I$, then denote the resulting ``pinched'' lasagna filling still by $\beta$.
    
    Define $\mathcal{S}_Y(\beta)$ to be the bimodule homomorphism from $\mathcal{S}_Y(b)$ to $\mathcal{S}_Y(b')$, induced by concatenating the pinched $\beta$ on top of the lasagna fillings in $Y \times I$; see \autoref{fig:3-manifold-bimodule} (right).
\end{itemize}
\end{definition}

\begin{figure}[ht!]
\centering
    \begin{minipage}[b]{0.3\linewidth}
    \includegraphics[width=\linewidth]{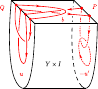}
    \end{minipage}
    \hspace{4em}
    \begin{minipage}[b]{0.3\linewidth}
    \includegraphics[width=\linewidth]{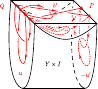}
    \end{minipage}
    \caption{
    Left: for a tangle $b$, the bimodule $\mathcal{S}_Y(b')$ sends a pair of tangles $u, u'$ to the skein lasagna module $\mathcal{S}(Y \times I, (u \times \{-1\}) \cup b \cup (-u' \times \{1\}))$. Right: for a lasagna filling $\beta$ in $\Sigma \times I \times I$, the bimodule homomorphism $\mathcal{S}_Y(\beta)$ is induced by stacking the pinched $\beta$ on top of the lasagna fillings in $Y \times I$.
\label{fig:3-manifold-bimodule}}
\end{figure}

If $\iota \colon \Sigma \to \partial Y$ is an orientation-reversing diffeomorphism, then we obtain a left $\mathcal{S}(\Sigma)$-module $S_Y \colon \mathcal{S}(\Sigma) \to \mathfrak{C}at$ instead. Again, if the diffeomorphism $\iota$ has no ambiguity in context, we will write $\partial Y = \Sigma$ if $\iota$ is orientation-preserving, write $\partial Y = -\Sigma$ if $\iota$ is orientation-reversing.

To formulate a gluing formula for $3$-manifold categories, next we define the tensor product of ${\mathcal S}(\Sigma)$-modules.
Recall from the last item of \autoref{def:3manifold module} that a $2$-cell $\beta$ in $\mathcal{S}(\Sigma)^{op}$ from $b$ to $b'$ induces (by vertical gluing, as shown in \autoref{fig:3-manifold-bimodule} (right)) 
maps
$$\mathcal{S}_{Y_1}(b)(u, u')\to \mathcal{S}_{Y_1}(b')(u, u'), \, F \mapsto F\cdot\beta, \;\,  {\rm and} \;\, \mathcal{S}_{Y_2}(b')(v, v')\to \mathcal{S}_{Y_2}(b)(v, v'), \, G \mapsto \beta\cdot G.
$$

\begin{definition} \label{def: tensor product 3manifolds}
    Given two compact $3$-manifolds $Y_1, Y_2$ with $\partial Y_1 = \Sigma$ and $\partial Y_2 = -\Sigma$, the tensor product $\mathcal{S}_{Y_1} \otimes_{\mathcal{S}(\Sigma)} \mathcal{S}_{Y_2}$ is the $1$-category defined as follows.

\begin{itemize}
\setlength{\itemsep}{5pt}
\item[-] The objects are 
         \[
         \Obj(\mathcal{S}_{Y_1} \otimes_{\mathcal{S}(\Sigma)} \mathcal{S}_{Y_2}) = \bigsqcup_{P \in \Obj({\mathcal{S}(\Sigma)})}\Obj(\mathcal{S}_{Y_1}(P)) \times \Obj(\mathcal{S}_{Y_2}(P)).\] 

\item[-] Given $u\in \Obj(\mathcal{S}_{Y_1}(P)), v\in  \Obj(\mathcal{S}_{Y_2}(P)), u'\in \Obj(\mathcal{S}_{Y_1}(Q)), v'\in  \Obj(\mathcal{S}_{Y_2}(Q))$ and a horizontal tangle $b$ from $P$ to $Q$, the morphism set ${\rm Hom}((u, v), (u', v'))$ in $\mathcal{S}_{Y_1} \otimes_{\mathcal{S}(\Sigma)} \mathcal{S}_{Y_2}$ is the graded ${\mathbb Z}$-module 
$$
{\rm Hom}((u, v), (u', v'))=\left[ \bigoplus_b  \mathcal{S}_{Y_1}(b)(u, u')\otimes_{\mathbb{Z}} \mathcal{S}_{Y_2}(b)(v, v')\right]\big/ \sim,
$$
where the equivalence relation $\sim$ is generated by
$ F\otimes (\beta\cdot G)\sim (F\cdot \beta)\otimes G
$, where 
 $F\in \mathcal{S}_{Y_1}(b)(u, u'), G\in \mathcal{S}_{Y_2}(b')(v, v')$, \autoref{fig:equivalence}.
\end{itemize}
\end{definition}

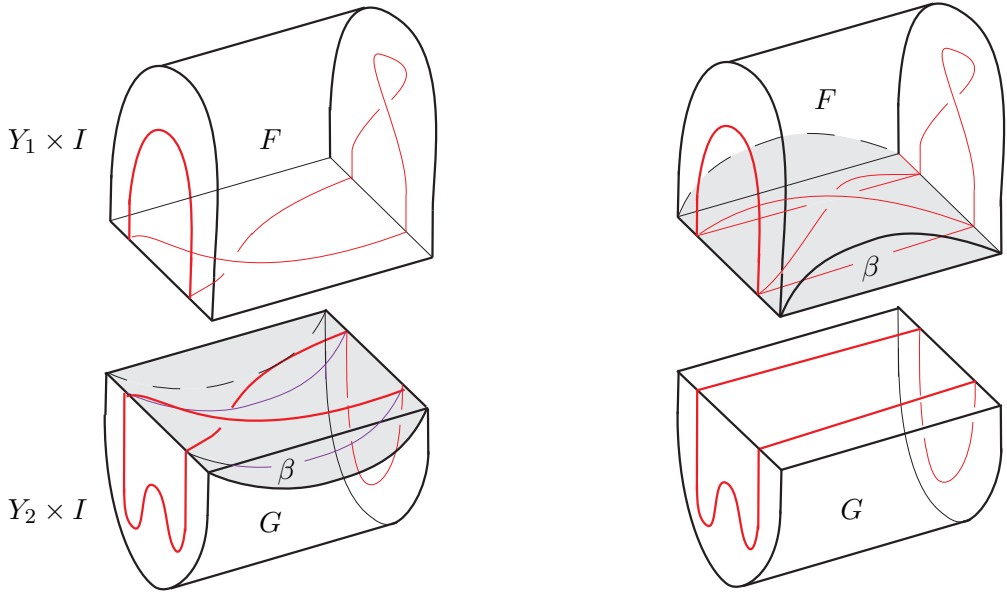
\begin{figure}[ht]
\includegraphics[height=8cm]{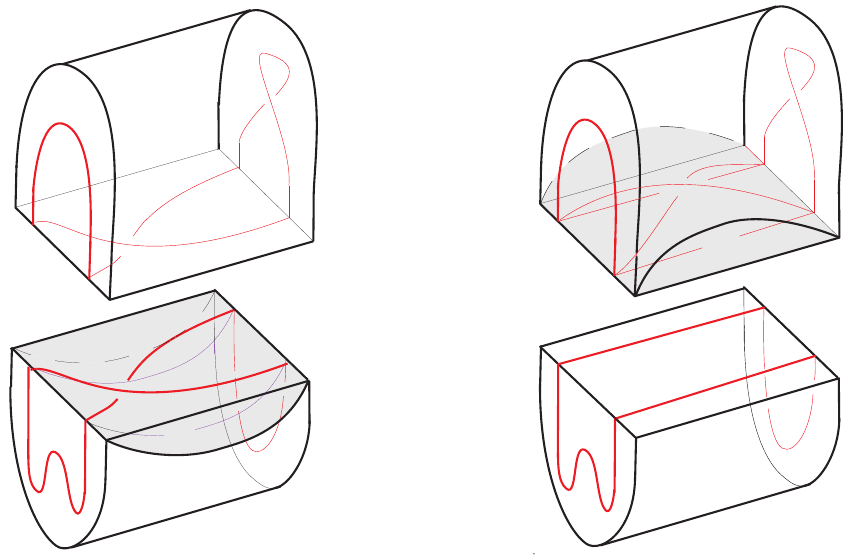}
{\small 
\put(-380,30){$Y_2\times I$}
\put(-380,170){$Y_1\times I$}
\put(-285,170){$F$}
\put(-285,25){$G$}
\put(-278,45){$\beta$}
\put(-75,185){$F$}
\put(-65,30){$G$}
\put(-57,122){$\beta$}
}
\caption{A schematic representation of the equivalence relation $ F\otimes (\beta\cdot G)\sim (F\cdot \beta)\otimes G
$. The ``pinched'' $2$-cell $\beta$ is shaded in the figure.}
\label{fig:equivalence}
\end{figure}

Throughout the rest of this section we will assume that the embedding $\Sigma\subset Y=Y_1\cup Y_2$ is fixed, while links in $Y$ and skein lasagna fillings  in $Y\times I$ can move by isotopies. Given a link $L\subset Y$ transverse to $\Sigma$, cut it along $\Sigma$ to get $(u,v)$, a pair of tangles in $Y_1, Y_2$ with common boundary points in $\Sigma$. In the proofs below we will consider links with various sub- and superscripts; by convention the corresponding tangles pairs will inherit the same decorations. For example, given $L', L^0$; the  corresponding pairs will be denoted $(u', v'), (u^0, v^0)$ respectively.

\begin{lemma} \label{lem: isomorphic objects}
    If $L, L'$ are isotopic links in $Y$ transverse to $\Sigma$, the resulting pairs $(u,v), (u',v')$ are isomorphic objects in $\mathcal{S}_{Y_1} \otimes_{\mathcal{S}(\Sigma)} \mathcal{S}_{Y_2}$.
\end{lemma}

\begin{proof}
    Consider the track of an isotopy $A:=(\sqcup S^1)\times I\subset Y\times I$. Applying an isotopy rel boundary if necessary, we may assume this collection of annuli $A$ is transverse to the fixed embedding $\Sigma\times I\subset Y\times I$. Let $L_t$ denote the slice  $(\sqcup S^1)\times \{ t\}$, $0\leq t\leq 1$, where $L_0=L$ and $L_1=L'$. 

    There is a finite collection of times $t_i$ when $L_t$ is tangent to $\Sigma$. Let $L'_i$ denote the link $L_{t_i-\varepsilon}$ and $L''_i$ the link $L_{t_i+\varepsilon}$. The intersection of the track of the isotopy with $Y_1\times I$ and $Y_2\times I$ gives skein lasagna fillings $F_i\subset Y_i\times I$, $i=1,2$. The equivalence class of the tensor product $[F_1\otimes F_2]$ gives morphisms $f\colon (u'_i,v'_i)\to (u''_i, v''_i)$, $g\colon (u''_i,v''_i)\to (u'_i, v'_i)$.

    The equivalence relation on morphisms in \autoref{def: tensor product 3manifolds} implies that $g\circ f={\rm Id}_{(u'_i,v'_i)}, f\circ g={\rm Id}_{(u''_i,v''_i)}$, \autoref{fig:equivalence1}.
    Away from the critical times $\{ t_i\}$, the isotopy of links does not interact with $\Sigma$, and the track of the isotopy provides isomorphisms $(u''_i,v''_i)\cong (u'_{i+1}, v'_{i+1})$.
\end{proof}
\begin{figure}[t]
\includegraphics[height=8cm]{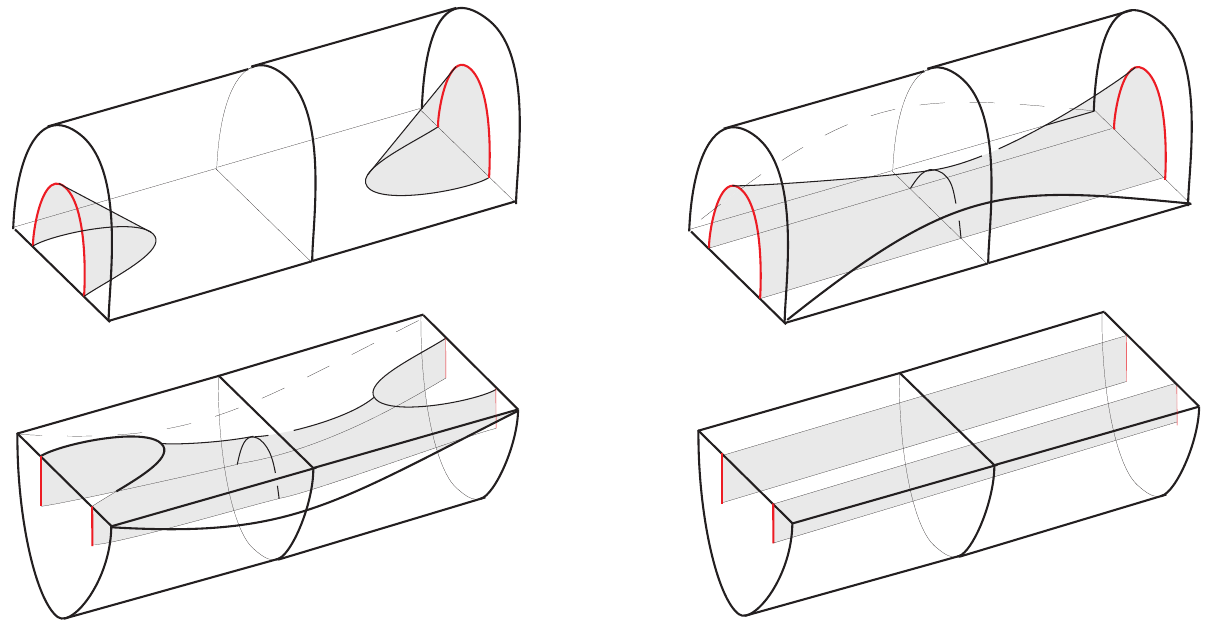}
\caption{Left: a local illustration of a morphism $g\circ f$. It is equivalent (moving a pinched $2$-cell from $Y_2\times I$ to $Y_1\times I$ as in as in \autoref{fig:equivalence}) to the morphism on the right, which is the identity morphism.}
\label{fig:equivalence1}
\end{figure}

\begin{remark}
    Different isotopies $L\cong L'$ may give rise to different isomorphisms $(u,v)\cong (u',v')$.
\end{remark}

\begin{theorem} \label{thm: gluing 3-manifolds}
    Suppose that $Y_1, Y_2$ are $3$-manifolds with $\partial Y_1 = \Sigma$ and $\partial Y_2 = -\Sigma$. Then
    $$
    \mathcal{S}_{Y} \cong \mathcal{S}_{Y_1} \otimes_{\mathcal{S}(\Sigma)} \mathcal{S}_{Y_2},
    $$
    where $Y=Y_1 \cup_{\Sigma} Y_2$.
\end{theorem}

\begin{proof}
We start by setting up a notation for links and their isotopies. 
For each link $L\subset Y$ consider a fixed representative $L_0$, transverse to $\Sigma$, in its isotopy class. Also fix an isotopy $H_L$ from $L$ to $L_0$ in $Y$. Note that by \autoref{lem: isomorphic objects}, $H_L$ gives rise to specific isomorphisms
\begin{equation} \label{eq:isos}
    f_{u,v}\colon (u,v)\mapsto (u_0, v_0), \; g_{u,v}\colon (u_0,v_0)\mapsto (u, v).
\end{equation}

Define the functors $$ \alpha\colon \mathcal{S}_{Y_1} \otimes_{\mathcal{S}(\Sigma)} \mathcal{S}_{Y_2}\to \mathcal{S}_{Y}, \, \beta\colon \mathcal{S}_{Y} \to \mathcal{S}_{Y_1} \otimes_{\mathcal{S}(\Sigma)} \mathcal{S}_{Y_2}.$$

The first functor $\alpha$
is given on objects and morphisms by gluing tangles and lasagna fillings. 
Now we define $\beta$. It sends an object (a link $L\subset Y$) to $(u_0, v_0)\in {\rm Obj}(\mathcal{S}_{Y_1} \otimes_{\mathcal{S}(\Sigma)})$. In more detail, $L_0$ is the fixed representative in the isotopy class of $L$, transverse to $\Sigma$, and $(u_0, v_0)$ is the results of cutting $L_0$ along $\Sigma$. Note that we could not cut $L$ along $\Sigma$ to obtain a pair of tangles, in general, because $L$ might not be transverse to $\Sigma$. To summarize,
\begin{equation} \label{eq: beta}
 {\beta}(L)=(u_0, v_0).
 \end{equation}
Next we define $\beta$ on morphisms. Let $F\in {\rm Hom}(L,L')$, that is $F$ is a skein lasagna filling in ${\mathcal S}(Y\times I, -L\cup L')$.
Define $\widetilde{F}:=H^{-1}_L\circ F\circ H_{L'}$, a skein lasagna filling in ${\mathcal S}(Y\times I, -L_0\cup L'_0)$ obtained by horizontal concatenation of the chosen isotopy $H^{-1}_L\subset Y\times I$ from $L_0$ to $L$, followed by $F$, followed by the isotopy $H_{L'}$ from $L'$ to $L'_0$. We may assume that $\widetilde{F}$ is transverse to $\Sigma\times I$ by applying an isotopy rel boundary. Cut along $\Sigma\times I$ to get $\widetilde{F}_1\otimes \widetilde{F}_2\in \mathcal{S}_{Y_1}(b)(u_0, u'_0)\otimes \mathcal{S}_{Y_2}(b)(v_0, v'_0)$, where $b:=\widetilde{F}\cap (\Sigma\times I)$.
An argument directly analogous to the proof of
\autoref{lem:lemma1} shows that different choices of isotopies making $\widetilde{F}$ transverse to $\Sigma\times I$ do not affect the class of $\widetilde{F}_1\otimes \widetilde{F}_2$ modulo the equivalence relation on morphisms in \autoref{def: tensor product 3manifolds}. 
Observe that $\beta$ is a functor:
\begin{itemize}
    \item[-] $\beta({\rm Id}_{(u,v)})={\rm Id}_{(u,v)},$
    \item[-] $\beta(F\circ G)=\beta(F)\circ \beta(G).$
\end{itemize}
Both of these statements follow from the fact that the concatenation of the isotopies $H^{-1}_L\circ H_L$ is isotopic to the identity isotopy of $L_0$.

The proof of \autoref{thm: gluing 3-manifolds} is concluded by observing that $\beta\circ\alpha$ is naturally isomorphic to the identity functor of $\mathcal{S}_{Y_1} \otimes_{\mathcal{S}(\Sigma)} \mathcal{S}_{Y_2}$ and that $\alpha\circ \beta$ is naturally isomorphic to the identity functor of ${\mathcal{S}_Y}$. 

To prove the first statement, let $(u,v)\in{\rm Obj}(\mathcal{S}_{Y_1} \otimes_{\mathcal{S}(\Sigma)} \mathcal{S}_{Y_2})$. Let $L$ denote the link $u\cup v$ in $Y$, then $\beta(\alpha(u,v))=(u_0,v_0)$ as in \autoref{eq: beta}. Moreover, by \autoref{eq:isos} we have fixed isomorphisms between $(u,v)$ and $(u_0, v_0)$. Then for any morphism $F\in {\rm Hom}((u, v), (u', v'))$, the following diagram commutes:
\[
\xymatrix{(u,v) \ar[r]^{F}\ar[d]^{f_{u,v}} & (u',v')  \\
\beta(\alpha((u,v))=(u_0,v_0) \ar[r]^{\beta(\alpha(F))} & \beta(\alpha((u',v'))=(u'_0,v'_0)\ar[u]_{g_{u',v'}} }
\]
because 
\[
g_{u',v'}\circ \widetilde{F}\circ f_{u,v}= H_L\circ H^{-1}_L\circ F\circ H_{L'}\circ H^{-1}_{L'}\cong F,
\]
where the last equivalence is an isotopy.
It follows that the family $\{ f_{u,v}\}$ provides a natural isomorphism between $\beta\circ\alpha$ and the identity functor.

The proof that $\alpha\circ \beta$ is naturally isomorphic to ${\rm Id}_{\mathcal{S}_Y}$ is analogous, using isomorphisms (lasagna fillings given by isotopies $H_L$) between the objects corresponding to isotopic links $L, L_0$. 
\end{proof}

\sloppy 
\printbibliography[title={Bibliography}]

@book {Handbook,
    AUTHOR = {Borceux, Francis},
     TITLE = {Handbook of categorical algebra. 1},
    SERIES = {Encyclopedia of Mathematics and its Applications},
    VOLUME = {50},
      NOTE = {Basic category theory},
 PUBLISHER = {Cambridge University Press, Cambridge},
      YEAR = {1994},
     PAGES = {xvi+345},
      ISBN = {0-521-44178-1},
   MRCLASS = {18-02 (18Axx)},
  MRNUMBER = {1291599},
MRREVIEWER = {Martin\ Hyland},
}

@book {JY,
    AUTHOR = {Johnson, Niles and Yau, Donald},
     TITLE = {{$2$}-Dimensional Categories},
 PUBLISHER = {Oxford University Press, Oxford},
      YEAR = {2021},
     PAGES = {xix+615}
}

@article {KR,
    AUTHOR = {Khovanov, Mikhail and Rozansky, Lev},
     TITLE = {Matrix factorizations and link homology},
   JOURNAL = {Fundamenta Mathematicae},
    VOLUME = {199},
      YEAR = {2008},
    NUMBER = {1},
     PAGES = {1--91},
      ISSN = {0016-2736,1730-6329},
   MRCLASS = {57M25 (57M27)},
  MRNUMBER = {2391017},
       DOI = {10.4064/fm199-1-1},
       URL = {https://doi.org/10.4064/fm199-1-1},
}

@article{MWW3,
  title={{Invariants of surfaces in smooth $4$-manifolds from link homology}},
  author={Morrison, Scott and Walker, Kevin and Wedrich, Paul},
  journal={arXiv preprint arXiv:2401.06600},
  year={2024}
}

@article {DM,
    AUTHOR = {Douglas, Christopher L. and Manolescu, Ciprian},
     TITLE = {On the algebra of cornered {F}loer homology},
   JOURNAL = {Journal of Topology},
    VOLUME = {7},
      YEAR = {2014},
    NUMBER = {1},
     PAGES = {1--68}
}

@article {DLM,
    AUTHOR = {Douglas, Christopher L. and Lipshitz, Robert and Manolescu,
              Ciprian},
     TITLE = {Cornered {H}eegaard {F}loer homology},
   JOURNAL = {Memoirs of the American Mathematical Society},
    VOLUME = {262},
      YEAR = {2019},
    NUMBER = {1266},
     PAGES = {v+124}
}

@book{Banyaga,
    AUTHOR = {Banyaga, Augustin},
     TITLE = {{The Structure of Classical Diffeomorphism Groups}},
    SERIES = {Mathematics and its Applications},
    VOLUME = {400},
 PUBLISHER = {Kluwer Academic Publishers Group, Dordrecht},
      YEAR = {1997},
     PAGES = {xii+197}
}

@incollection{FNWW,
    AUTHOR = {Freedman, Michael and Nayak, Chetan and Walker, Kevin and Wang, Zhenghan},
     TITLE = {On picture {$(2+1)$}-{TQFT}s},
 BOOKTITLE = {Topology and Physics},
    SERIES = {Nankai Tracts in Mathematics},
    VOLUME = {12},
     PAGES = {19--106},
 PUBLISHER = {World Scientific Publishing, Hackensack, NJ},
      YEAR = {2008}
}

@article{GK,
   title={Trisecting {$4$}-manifolds},
   author={Gay, David and Kirby, Robion},
   volume={20},   
   number={6},
   journal={Geometry \& Topology},
   publisher={Mathematical Sciences Publishers},
   year={2016},
   pages={3097--3132}
}

@article{MN,
    AUTHOR = {Manolescu, Ciprian and Neithalath, Ikshu},
     TITLE = {{Skein lasagna modules for $2$-handlebodies}},
   JOURNAL = {Journal f\"ur die Reine und Angewandte Mathematik. (Crelle's
              Journal)},
    VOLUME = {788},
      YEAR = {2022},
     PAGES = {37--76}
}

@article{MWW1,
    AUTHOR = {Manolescu, Ciprian and Walker, Kevin and Wedrich, Paul},
     TITLE = {Skein lasagna modules and handle decompositions},
   JOURNAL = {Advances in Mathematics},
    VOLUME = {425},
      YEAR = {2023},
     PAGES = {Paper No. 109071, 40}
}

@article{MWW2,
  title={Invariants of {$4$}-manifolds from Khovanov-Rozansky link homology},
  author={Morrison, Scott and Walker, Kevin and Wedrich, Paul},
  journal={Geometry \& Topology},
  volume={26},
  number={8},
  pages={3367--3420},
  year={2022}
}

@article{RW,
  title={{Khovanov homology and exotic $4$-manifolds}},
  author={Ren, Qiuyu and Willis, Michael},
  journal={arXiv preprint arXiv:2402.10452},
  year={2024}
}

@article{SZ,
  title={Kirby belts, categorified projectors, and the skein lasagna module of {$S^2\times S^2$}},
  author={Sullivan, Ian A. and Zhang, Melissa},
  journal={arXiv preprint arXiv:2402.01081},
  year={2024}
}

@misc{Walker,
    author = {Walker, Kevin},
    title = {TQFTs notes},
    howpublished = {Available at \url{http://canyon23.net/math/}},
    year = {2006}
}

@article{MW,
    AUTHOR = {Morrison, Scott and Walker, Kevin},
     TITLE = {Blob homology},
   JOURNAL = {Geometry \& Topology},
    VOLUME = {16},
      YEAR = {2012},
    NUMBER = {3},
     PAGES = {1481--1607}
}

@book{Wall,
    AUTHOR = {Wall, C. T. C.},
     TITLE = {Differential {T}opology},
    SERIES = {Cambridge Studies in Advanced Mathematics},
    VOLUME = {156},
 PUBLISHER = {Cambridge University Press, Cambridge},
      YEAR = {2016},
     PAGES = {viii+346}
}

\end{document}